\documentclass[12pt,twoside,reqno,final]{amsart}

\usepackage{microtype}
\usepackage[OT1]{fontenc}
\usepackage{textcomp}
\usepackage{type1cm}
\usepackage{amssymb}
\usepackage{mathtools}
\usepackage{esint}
\usepackage{nicefrac}
\usepackage[left=3cm,top=4cm,right=3cm]{geometry}

\usepackage[active]{srcltx}
\usepackage{verbatim}

\usepackage[pdftex]{graphicx}
\usepackage{caption}
\usepackage{subfigure}
\usepackage{color}
\usepackage{tikz}
\usetikzlibrary{trees,patterns,calc}

\numberwithin{equation}{section}

\theoremstyle{plain}
\newtheorem{theorem}{Theorem}[section]

\newtheorem{corollary}[theorem]{Corollary}
\newtheorem{proposition}[theorem]{Proposition}
\newtheorem{lemma}[theorem]{Lemma}

\theoremstyle{remark}
\newtheorem{remark}[theorem]{Remark}

\theoremstyle{definition}

\newcommand{\HH}{\mathcal{H}}

\newcommand{\CC}{\mathcal{C}}

\newcommand{\R}{\mathbb{R}}

\newcommand{\N}{\mathbb{N}}

\newcommand{\roo}{\varrho}

\renewcommand{\epsilon}{\varepsilon}
\newcommand{\eps}{\varepsilon}
\newcommand{\e}{\varepsilon}

\DeclareMathOperator{\dir}{dir}

\DeclareMathOperator{\dimb}{dim_B}
\DeclareMathOperator{\udimb}{\overline{dim}_B}

\DeclareMathOperator{\dimh}{dim_H}

\DeclareMathOperator{\dimp}{dim_P}

\DeclareMathOperator{\dist}{dist}
\DeclareMathOperator{\diam}{diam}

\begin{document}

\title{H\"older coverings of sets of small dimension}

\author{Eino Rossi}
\author{Pablo Shmerkin}
\address{
        Department of Mathematics and Statistics, Torcuato di Tella University \\
        Av. Figueroa Alcorta 7350 (C1428BCW) \\
        Buenos Aires \\
        Argentina}
\email{eino.rossi@gmail.com}
\email{pshmerkin@utdt.edu}
\thanks{ER acknowledges the supports of CONICET and the Finnish Academy of Science and Letters}
\thanks{PS was supported by PICT 2013-1393 and PICT 2014-1480 (ANPCyT)}
\date{\today}

\begin{abstract}
 We show that a set of small box counting dimension can be covered by a H\"older graph from all but a small set of directions,  and give sharp bounds for the dimension of the exceptional set, improving a result of B. Hunt and V. Kaloshin. We observe that, as a consequence, H\"older graphs can have positive doubling measure, answering a question of T. Ojala and T. Rajala. We also give remarks on H\"older coverings in polar coordinates and, on the other hand, prove that a Homogenous set of small box counting dimension can be covered by a Lipschitz graph from all but a small set of directions.
\end{abstract}

\maketitle
\section{Introduction}

Given a set $A\subset\R^d$, how often does the orthogonal projection to a $k$-plane $V$ have a H\"{o}lder inverse? Of course, in order to have an inverse at all, the projection has to be injective. Let $\udimb$ denote upper box dimension. It follows from elementary dimension inequalities that if $\udimb(A)<(d-1)/2$, then for almost all $v\in S^{d-1}$, the orthogonal projection $P_v:\R^d\to \langle v\rangle^\perp$ is indeed injective. In \cite{HuntKaloshin1999}, Hunt and Kaloshin proved that, in this case, for almost all $v\in S^{d-1}$, the set $A$ can be covered by the graph of a H\"{o}lder function $f_v:\langle v\rangle^\perp\to \langle v\rangle$. More generally, they showed that if $\udimb(A)<(d-k)/2$, then for almost all linear maps $L$ from $\R^d\to\R^{d-k}$, the restriction $L|_A$ has a H\"{o}lder inverse. Hunt and Kaloshin also obtain results for subsets $A$ of infinite-dimensional spaces, where ``almost every'' is understood in the sense of prevalence. See \cite{HuntKaloshin1999} for further details.

In this article, we refine Hunt and Kaloshin's result by providing a bound on the dimension of exceptional directions $v$, and likewise for $k$-planes $V$ for any $k$; see Theorem \ref{thm:holder}. Moreover, we show that when $k=1$ this bound is sharp in a rather strong way. While Hunt and Kaloshin state their results in terms of almost all linear maps, our approach is to work with the Grassmannian of $k$ planes in $\R^d$, which is perhaps a more natural parametrization, since many linear maps correspond geometrically to the same projection. We also establish analogous results for spherical projections, and observe that for a special class of sets (homogeneous sets) one can get covers by \emph{Lipschitz} graphs, while this is known to be false in general.

One of the original motivations for this work was a question posed by T. Rajala and T. Ojala in \cite{OjalaRajala2014}: does every doubling measure on $\R^2$ give zero mass to a H\"{o}lder graph? It turns out that a strong negative answer follows by combining the aforementioned result of Hunt and Kaloshin with some known constructions of sets of small box dimension which are charged by a doubling measure. Indeed, we observe that there are doubling self-similar measures that charge graphs of H\"{o}lder functions of exponent arbitrarily close to $1$. We were initially unaware of Hunt and Kaloshin's result and an earlier version of this article contained an independent derivation; we thank M. Hochman for bringing the work \cite{HuntKaloshin1999} to our attention.

\section{Covers of small sets by H\"{o}lder graphs}

\subsection{A bound on the dimension of exceptional planes}

Let us first fix some notation and definitions. Let $\dimh,\dimp$ and $\udimb$ denote the Hausdorff, packing and upper box counting (or Minkowski) dimensions respectively. For the definitions and main properties, see  for example \cite{Mattila1995}.

We let $G(d,k)$ be the Grassmannian of $k$-planes in $\R^d$. This is a compact smooth manifold of dimension $k(d-k)$. A natural metric compatible with the topology of $G(d,k)$ is given by $\roo(V,W)=\|P_V-P_W\|$, where $P_{\cdot}$ denotes parallel projection (i.e. orthogonal projection to the orthogonal complement of the plane), and $\|\cdot\|$ stands for the operator norm. The Hausdorff and box counting dimensions of $G(d,k)$ in this metric are again $k(d-k)$. We note that the orthogonal group $O(d)$ acts transitively on $G(d,k)$, and that the metric $\roo$ is invariant under this action. Moreover, $G(d,k)$ carries a unique Borel probability measure $\gamma_{d,k}$ invariant under this action. The Grassmannian $G(d,1)$ can be naturally identified with the $(d-1)$ dimensional projective space.  For further details about $G(d,k)$ and $\gamma_{d,k}$, the reader is referred to  \cite[Section 3]{Mattila1995}.

For $w\ge 0$, let $\HH_\infty^w$ denote the $w$ dimensional Hausdorff content on the Grassmanian $G(d,k)$ (with respect to the metric $\roo$ defined above). That is, for $E\subset G(d,k)$
\[
 \HH_\infty^w(E) = \inf\left\{ \sum_{i\in\N} \diam(U_i)^w : E \subset \bigcup_{i\in\N} U_i \right\}.
\]
Hausdorff content is an outer measure, but it is in general highly non-additive even on Borel sets. One exception is the value $w=k(d-k)$; in this case $\gamma_{d,k} = c\, \HH_\infty^{k(d-k)}$, for some constant $c$ depending on $d$ and $k$. We recall that $\dimh(E) = \inf \{ w>0: \HH_\infty^w(E)=0\}$.

We can now state our main theorem:
\begin{theorem}
\label{thm:holder}
 Let $A\subset\mathbb{R}^d$ be a bounded set such that $\udimb(A)<t< (d-k)/2$ and let $(k-1)(d-k) + 2t < w < k(d-k) $. Then the set of planes $V\in G(d,k)$, for which the set $A$ is not contained in the graph of a H\"older function $f_V\colon V^\perp\to V$ of exponent $1-\frac{2t}{w-(k-1)(d-k)}$, has Hausdorff dimension at most $w$.

 In particular, the set of $V\in G(d,k)$ such that $A$ is not contained in the graph of a H\"older function $f_V\colon V^\perp\to V$ (without taking exponents into account) has Hausdorff dimension at most $(k-1)(d-k) + 2\udimb(A)$.
\end{theorem}

In the proof we will use the following result from elementary geometry. We use the notation $x=O(r)$ to mean $0\le x\le Cr$, where the constant $C$ may depend only on the ambient dimension and $w$, the exponent of the Hausdorff content in question.
\begin{lemma} \label{lem:anglegeneral}
Let $B, B'$ be two balls in $\R^d$ of radius $r$ that are at distance $R\ge r$ apart, and let $\ell\in G(d,1)$ be the direction determined by their centres. Then any direction determined by points $x\in B$ and $x'\in B'$ makes an angle at most $O(r/R)$ with $\ell$.
\end{lemma}
\begin{proof}[Proof of Theorem \ref{thm:holder}]
 For simplicity, denote $(k-1)(d-k)=g$, and note that this is the dimension of $G(d-1,k-1)$. By assumption, there are $s<t$, a constant $C>0$ and families $\mathcal{B}_n$ of balls of radius $2^{-n}$, such that $|\mathcal{B}_n|\le C\, 2^{sn}$ and $A$ is covered by the union of $\mathcal{B}_n$, for each $n\in\N$.
 Let
 \[
  \mathcal{C}_n = \left\{ (2B,2B'): (B,B')\in\mathcal{B}_n^2, \,\dist(B,B')\ge 2\cdot 2^{-(1-\frac{2t}{w-g}) n} \right\}.
 \]
 Here $2B$ denotes the ball of the same center as $B$ and twice the radius. Given $\ell\in G(d,1)$ and $V\in G(d,k)$, let $\angle(\ell,V)$ denote the respective angle, i.e. the infimum of the angles between non-zero vectors in $\ell$ and $V$.  Let us define
 \[
  H_{B,B'}(\delta) = \{ V\in G(d,k) : \angle(\ell(B,B'),V)\leq \delta \},
 \]
 where $\ell(B,B')\in G(d,1)$ has the direction determined by the centers of the balls $B$ and $B'$. By Lemma \ref{lem:anglegeneral}, the set of $k$-planes which contain a direction determined by two points in $B, B'$, $(B,B')\in\CC_n$  is then contained in
 \begin{equation}
  \label{eq:lastobservation}
  H_{B,B'} \left( O(1)\frac{ 2^{-n} }{  2^{-n(1-\frac{2t}{w-g}) } } \right)
  = H_{B,B'} \left( O(1) 2^{-n(\frac{2t}{w-g})} \right).
 \end{equation}

 We need to estimate the $w$-dimensional Hausdorff content of the set $H_{B,B'}(\delta)$ for small $\delta$. The idea is that if $V\in H_{B,B'}(\delta)$, then  $V$ contains a vector that makes a small angle with $\ell(B,B')$, which we can complete to an orthonormal basis of $V$; the degrees of freedom then equal the dimension of $G(d-1,k-1)$. To make this idea precise, let $B(V_i,\delta)$ be a cover of $G(d-1,k-1)$ with cardinality $O(1) \delta^{-g}$. Let $\varphi$ be any orthogonal map from $\R^{d-1}$ to $\ell(B,B')^\perp\subset\R^d$, and given $V\in G(d-1,k-1)$ let $\hat{V}\in G(d,k)$ be the plane spanned by $\varphi(V)$ and $\ell(B,B')$. Note that $V\mapsto \hat{V}$ is an isometry onto its image. Furthermore, if $W\in H_{B,B'}(\delta)$, then there are $V\in G(d-1,k-1)$ and an orthogonal map $\psi\in O(d)$ with $\|\psi-I\|=O(\delta)$ such that $W=\psi\hat{V}$. We deduce that
 \[
  H_{B,B'}(\delta) \subset \bigcup_{i} B(\hat V_i,O(\delta)),
 \]
 and therefore
 \begin{equation} \label{eq:contentbound}
  \HH_\infty^w (H_{B,B'}(\delta)) \leq O(1) \delta^{-g} \delta^w
 \end{equation}
 directly from the definition of $\HH_\infty^w$. Let $M_n\subset G(d,k)$ be the planes which contain a direction determined by two points in $B, B'$ for some $(B,B')\in\CC_n$. Using equations \eqref{eq:lastobservation} and \eqref{eq:contentbound}, and the obvious bound $|\mathcal{C}_n |\le |\mathcal{B}_n|^2 \le C^2 2^{2sn}$, we estimate
 \begin{align*}
  \HH_\infty^w( M_n )
  &\leq \sum_{(B,B')\in\CC_n} \HH_\infty^w\left( H_{B,B'}\left( O(1) 2^{-n2t/(w-g)}\right) \right) \\
  &\leq O(C^2) 2^{n2s} 2^{-n2t}
   = O(C^2) 2^{2n (s-t)}
 \end{align*}
 Since $s<t$, we have that $\sum_{n\in\N} \HH_\infty^w\left( M_n \right) < \infty$. Hence the set of $k$-planes which are in infinitely many $M_n$ has $\HH_\infty^w$ measure zero, and hence Hausdorff dimension at most $w$. 

 It remains to show that if $V$ is only in finitely many of the $M_n$, then $A$ can be covered by the graph of a suitable function $f:V^\perp \to V$. Fix, then, some large $n_0\in\N$ and $V\in G(d,k) \setminus \cup_{n=n_0}^\infty M_n$, and write $P$ for the orthogonal projection to $V^{\perp}$. Now let $x,x'\in A$ and suppose
 \[
  |x-x'| \ge 3\cdot 2^{- (1-\frac{2t}{w-g}) n}
 \]
 for some $n\ge n_0$. Then $x\in B, x'\in B'$ for some $B,B'\in\mathcal{B}_n$, and the projections $P(2B), P(2B')$ are disjoint (otherwise, there would be points $y,y'\in 2B, 2B'$ determining a direction contained in the $k$-plane $V$, contradicting that $V\notin M_n$). This implies that
 \[
  |P(x)-P(x')|\ge 2\cdot 2^{-n},
 \]
 thus in particular $P$ is injective on $A$, and therefore $A$ is the graph of a function $f:P(A)\to V$. To show that $f$ is H\"older, let $P(x),P(x')\in P(A)$ so that $2^{-n} \leq |P(x)-P(x')| < 2\cdot 2^{-n}$ for some $n\ge n_0$. By the above observation,
 \begin{align*}
  |f(P(x))-f(P(x'))|&=  |x-x'| \le 3\cdot 2^{- (1-\frac{2t}{w-g}) n} \\
  &= 3 (2^{-n})^{1-\frac{2t}{w-g}}\\
  & \leq 3 ( |P(x)-P(x')| )^{1-\frac{2t}{w-g}}
 \end{align*}
 This estimate holds for all $n\geq n_0$ (that is, when $|P(x)-P(x')|$ is small), so $f$ is H\"{o}lder continuous with exponent $\alpha=(1-\frac{2t}{w-g})$, and a constant $C$ depending on $n_0$ and $w$.

 Finally, if $f=(f_1,\ldots,f_k)$ where $f_i$ has $\alpha$-H\"{o}lder constant $C_i$, then we extend $f$ to a H\"{o}lder function $\widetilde{f}$ on all of $V^\perp$ as follows: $\widetilde{f}=(\widetilde{f}_1,\ldots,\widetilde{f}_k)$, where
 \[
  \widetilde{f}_i(y) = \inf_{z\in A} \{ f_i(z) + C_i |y-z|^{\alpha} \}.
 \]
 The existence of such an extension is classical: for Lipschitz functions this is the McShane-Whitney extension theorem; the  H\"{o}lder case follows by applying this to a suitable snowflaking of the metric. 
\end{proof}

\subsection{Sharpness of Theorem \ref{thm:holder}}

Theorem \ref{thm:holder} is sharp in a number of ways. One cannot replace ``H\"{o}lder'' by ``Lipschitz'', since the graph of a Lipschitz function is porous, while a set of small box counting dimension needs not be porous. Box-counting dimension cannot be replaced by packing (or Hausdorff) dimension, since a dense countable set cannot be the graph of a H\"{o}lder function, yet has zero packing (and Hausdorff) dimension and Hausdorff dimension can not be changed to packing dimension in measuring the size of exceptions, see Lemma \ref{lem:dense-Gdelta-exceptions}. We will also show that the dimensional threshold on $A$ is also sharp, at least when $k=1$; this depends on the following construction:

\begin{proposition} \label{prop:perc}
For any $t\in[0,d]$, there exists a compact set $A\subset\R^d$ with $\dimh(A)=\dimb(A)=t$ such that the direction set
\[
\dir(A)=\{(x-y)/|x-y|:x\neq y\in A\}
\]
satisfies $\dimh(\dir(A))=\min(2t,d-1)$. Moreover, if $t>(d-1)/2$, then $A$ can be chosen so that $\dir(A)=S^{d-1}$.
\end{proposition}

The set $A$ can be taken as the $t$-dimensional fractal percolation limit set in $\R^d$. Before proceeding with the proof, we recall its construction. Let $\mathcal{Q}_n^d$ denote the closed dyadic cubes of side length $2^{-n}$ in $\R^d$, and let $\mathcal{Q}^d=\cup_n \mathcal{Q}_n^d$. Given $p\in (0,1)$, we define a decreasing sequence of closed subsets of $[0,1]^d$ as follows. Let $A_0=[0,1]^d$, keep each cube $Q\in\mathcal{Q}_1^d$ with probability $p$ with all the choices independent, and let $A_1$ be the union of the retained cubes. Now suppose $A_n$ has been defined as a union of cubes in $\mathcal{Q}^d_n$. Keep each cube $Q$ in $\mathcal{Q}_{n+1}^d$, $Q\subset A_n$ with probability $p$, with all choices independent of each other, and of previous stages of the construction, and let $A_{n+1}$ be the union of all selected cubes. Finally, we define $A=\cap_n A_n$.

It is well known that if $p=p_t=2^{t-d}$ for some $t\in (0,d]$ then, conditional on $A\neq \varnothing$, the Hausdorff and box counting dimension of $A$ is equal to $t$ (on the other hand, if $p\le 2^{-d}$ then $A$ is empty almost surely).

If we let $\nu_n = p^{-n}\mathbf{1}_{A_n}dx$, then almost surely $\nu_n$ converges weakly to a measure $\nu$ supported on $A$, known as the \emph{natural (fractal percolation) measure}. We refer the reader to \cite{ShmerkinSuomala2017} and references there for further background on fractal percolation.

\begin{proof}[Proof of Proposition \ref{prop:perc}]
Fix $t\in (0,d)$, and let $A$ be the fractal percolation limit set constructed with probability $p=2^{t-d}$ (and hence of box counting dimension $t$).

If $t>(d-1)/2$, then it was shown in \cite[Corollary 5.9]{ShmerkinSuomala2017} (see also \cite[Remark 5.10]{ShmerkinSuomala2017}) that $\dir(A)=S^{d-1}$. Hence we assume that $t\le (d-1)/2$ for the rest of the proof.

Given $v\in G(d,1)$, we denote
\[
W_v= \{ (x,y)\in\R^d\times\R^d: y-x\in\langle v\rangle \}\in G(2d,d+1).
\]
Write $G_{\pi/8}(d,1)$ for the lines in $G(d,1)$ making an angle $\ge \pi/8$ with all coordinate hyperplanes (the value $\pi/8$ is not important, any positive number will do). Let  $\Gamma_0=\{ W_v:v\in G_{\pi/8}(d,1)\}$. Further, let $\Gamma$ be the collection of all translations of planes in $\Gamma_0$ which hit the unit cube (this is a family of \emph{affine} subspaces of $\R^{2d}$).

It easy to check that $\Gamma$ satisfies the assumptions of \cite[Theorem 5.2 and Corollary 5.8]{ShmerkinSuomala2017} (we work with $G_{\pi/8}(d,1)$ instead of $G(d,1)$ to ensure transversality with respect to coordinate hyperplanes). Note that, in our case, $m=2$ and $k=d+1$. Fix two disjoint dyadic cubes $Q_1,Q_2$ of the same level (the level itself is not important), such that any line joining a point of $Q_1$ with a point of $Q_2$ makes an angle $\ge \pi/8$ with the coordinate hyperplanes. Further, let $\gamma>2(d-t)-(d-1)=d+1-2t$. By \cite[Corollary 5.8]{ShmerkinSuomala2017} applied to $U=Q_1\times Q_2$, and the definition of the natural measure, almost surely
\begin{equation} \label{eq:application-SS}
\sup_{n\in \N, W\in\Gamma} 2^{-n\gamma} 2^{2n(d-t)}\mathcal{H}^{d+1}\left( \left(A_n\times A_n\right)\cap (Q_1\times Q_2) \cap W \right) <\infty
\end{equation}
Write $A'_n=(A_n\times A_n)\cap (Q_1\times Q_2)$ for simplicity. By Fubini's Theorem and \eqref{eq:application-SS}, there is a (random) constant $K>0$ such that
\begin{align*}
\mathcal{H}^{2d}\left( A'_n\cap W(2^{-n})\right) &\le  \int_{B(0, 2^{-n})\subset W^\perp} \mathcal{H}^{d+1}(A'_n\cap (W+u)) \, du \\
&\le 2^{-n(d-1)}\, K 2^{n(\gamma -2(d-t))}\\
&=  K 2^{-2nd} 2^{n(\gamma-(d+1-2t))}   =: K 2^{-2nd} 2^{\delta n}
\end{align*}
for any $W\in \Gamma_0$ (here $W(\eps)$ denotes the $\eps$-neighborhood of $W$). Note that $\delta$ can be made arbitrarily small, and that $K$ depends on $\gamma$ (hence on $\delta$). Since any $Q\in\mathcal{Q}_n^{2d}$ that intersects $W$ is contained in $W(\sqrt{2d}2^{-n})$, comparing volumes we deduce that
\[
\#\{ Q\in\mathcal{Q}_n^{2d}: Q\cap W_v\neq\varnothing\} \le  K'_\delta \, 2^{\delta n}
\]
for every $v\in G_{\pi/8}(d,1)$ and some new random constant $K'_\delta>0$. On the other hand, since $Q_1$ and $Q_2$ are separated, if $v,v'\in G(d,1)$ are at distance $\le 2^{-n}$, then
\[
W_{v'}\cap (Q_1\times Q_2)\subset (W_v\cap ( Q_1\times Q_2) )(C 2^{-n}),
\]
where $C>0$ is an absolute constant. Let $\dir(x,y)\in G(d,1)$ be the direction determined by $x\neq y\in \R^d$.  Combining the last two displayed equations, we deduce that for any $2^{-n}$-ball $B$ in $G(d,1)$ whose centre lies in $G_{\pi/8}(d,1)$,
\begin{equation} \label{eq:fiber-cover}
\#\{ Q\in\mathcal{Q}_n^{2d}: Q\subset Q_1\times Q_2, \dir^{-1}(B)\cap Q\neq\varnothing \} \le K''_\delta 2^{\delta n}
\end{equation}
for some new constant $K''_\delta$.

Now, conditioned on $Q_1$ and $Q_2$ being chosen, $A':=(A\times A)\cap (Q_1\times Q_2)$ has (up to scaling and translation) the distribution of the product of two independent copies of fractal percolation with the same parameter $p$. Hence, there is a positive probability that $A'$ has Hausdorff (and box) dimension $2t$. Fix a realization such that this happens, and such that \eqref{eq:fiber-cover} holds for all $\delta>0$ (this is possible since for fixed $\delta$ this is an almost sure event, and we only need to consider a sequence $\delta_j\downarrow 0$).

We can now finish the proof. Define $\Delta\colon \R^d\times \R^d \to G(d,1)$ by $\Delta(x,y)=(x-y)/|x-y|$, and note that $\dir(A)=\Delta( A\times A\setminus\{x=y\} )$. Since $\Delta$ is locally Lipschitz outside of the diagonal $\{x=y\}$, and packing dimension behaves well in products \cite[Theorem 8.10]{Mattila1995}, we know that
\[
\dimh(\dir(A)) \le \dimp(\dir(A)) \le \dimp(A\times A) \le 2\dimp(A)=2t,
\]
so the task is to establish the lower bound.

Let $\{B_j\}$ be a cover of $\Delta(A')$, where $B_j$ is a ball of radius $2^{-n_j}$, with all $n_j$ sufficiently large. Since $Q_1\times Q_2\subset\dir^{-1}(G_{\pi/8}(d,1))$ by our choice of $Q_1, Q_2$, there is a cover of $A'$ consisting, for each $j$, of $K'''_\delta 2^{\delta n_j}$ balls of radius $2^{-n_j}$. Hence
\[
\sum_j K'''_\delta 2^{\delta n_j} 2^{(2t-\delta) n_j} \ge 1 .
\]
This, however, shows that $\Delta(A')\subset\dir(A)$ has Hausdorff dimension $\ge 2t-2\delta$ which, letting $\delta\to 0$, completes the proof.
\end{proof}

\begin{remark}
We believe that fractal percolation should witness the sharpness of Theorem  \ref{thm:holder} for all values of $k$, but the proof given above strongly uses that $k=1$. We note, however, that it follows directly from the methods from \cite{ShmerkinSuomala2017} that if $t>(d-k)/2$, then the $t$-dimensional fractal percolation set $A$ has the property that for every $V\in G(d,k)$ there are points $x,y\in A$ with $y-x\in V$; see \cite[Remark 5.10]{ShmerkinSuomala2017}. Thus in this case all directions are exceptional in Theorem \ref{thm:holder}.
\end{remark}

\begin{corollary} \label{cor:sharpness}
For any $t\in (0,(d-1)/2]$, there exists a compact set $A$ with $\dimh(A)=\dimb(A)=t$ such that
\[
\dimh\{V\in G(d,1): P_V|_A \text{ is not injective } \} = 2t.
\]
\end{corollary}


Finally, we show that the exceptional set of planes can be (and often is) a dense $G_\delta$ subset of $G(d,k)$. Since dense $G_\delta$ subsets of complete metric spaces have full packing dimension, this also shows that in Theorem \ref{thm:holder} one cannot hope to measure the dimension of the exceptional set by packing dimension.
\begin{lemma} \label{lem:dense-Gdelta-exceptions}
Let $0<k<d$. If the direction set of $A\subset\R^d$ is dense in $\R^d$, then
\[
 E = \{ V\in G(d,k): P_V|_A \text{ has no H\"{o}lder inverse }\}
\]
is a dense $G_\delta$ set.

In particular, there are self-similar sets $A$ of all dimensions satisfying this, as well as compact countable sets.
\end{lemma}
\begin{proof}
Define a sequence of open sets
\begin{align*}
 E_N = \{ V\in G(d,k) :\  &\|x-y\| > N \|P_V(x)-P_V(y)\|^{1/N} \text{ for some } x,y\in A \\
     &\text{with } \|P_V(x)-P_V(y)\| < 1 \}.
\end{align*}
If $V\in G(d,k)$ contains a direction determined by $A$, then $V\in E_N$ for all $N$. Since $\dir(A)$ is dense in $S^{d-1}$, we have that each $E_N$ is dense in $G(d,k)$ (given a basis $(w_1,\ldots,w_k)$ for $W\in G(d,k)$, we find $v\in\dir(A)$ arbitrarily close to $w_1$, so that  $(v,w_2,\ldots,w_k)$ spans a $k$-plane in $E_N$, for all $N$, which is close to $W$). It is clear that $E=\cap_{n=1}^{\infty} \cup_{N=n}^{\infty} E_N=\cap_{N=1}^{\infty} E_N$ which is a dense $G_\delta$ set by Baire's theorem.

If $A$ is a self-similar set such that the linear parts of the similarities generate a dense subgroup of $O(d)$, then it is clear that $A$ spans a dense set of directions. Finally, if $\{ e_j\}$ is any dense subset of $S^{d-1}$, then $\{0\}\cup \{ 2^{-j}e_j\}$ is a compact countable set whose direction set is also dense.
\end{proof}
As a corollary, we deduce the following dichotomy:
\begin{corollary}
\label{cor:dichotomy}
For any set $A\subset\R^d$, one of the following alternatives holds:
\begin{enumerate}
\item\label{enu:joko} There is a dense $G_\delta$ subset $E$ of $G(d,1)$ such that for each $v\in E$, there is no H\"{o}lder function $f_v\colon\langle v\rangle^\perp \to \langle v\rangle$ whose graph covers $A$.
\item\label{enu:tai} There is a nonempty open set $U\subset G(d,1)$ such that for each $v\in U$, there is a Lipschitz function $f_v\colon\langle v\rangle^\perp \to \langle v\rangle$ whose graph covers $A$.
\end{enumerate}
\end{corollary}
\begin{proof}
We have seen in Lemma \ref{lem:dense-Gdelta-exceptions} that if $\dir(A)$ is dense in $S^{d-1}$ then the first alternative holds. Suppose then that the direction set is not dense. Then there are a nonempty open set $U\subset S^{d-1}$ and $\e>0$ such $\dist(U,\dir(A))>\eps$. This means that for all points $x\in A$ and $v\in U$, there is a cone in direction $v$, with opening angle $\eps$, which does not contain any other point of $A$. It follows (see e.g. the proof of \cite[Lemma 15.13]{Mattila1995}) that the inverse of the orthogonal projection of $A$ to $\langle v \rangle^\perp$ is Lipschitz, and this can be extended to a Lipschitz function on all of $\langle v \rangle^\perp$.
\end{proof}

\section{Lipschitz coverings and spherical projections}

As discussed in \S 2.2, one cannot hope to find coverings by Lipschitz graphs in the general case. However, in Corollary \ref{cor:dichotomy} we noted that Lipschitz covers do exist when the direction set is not dense. In Proposition \ref{prop:lip_inverse}, we show that for a rich class of non-trivial sets, not only there exists an open set of directions for which there is a Lipschitz covering of the set, but such covering exists in any direction that is not determined by the set.

\subsection{Lipschitz coverings of homogeneous sets}

We recall some definitions that originate in the work \cite{Furstenberg2008}. A set $M\subset\R^d$ is a \emph{miniset} of $A\subset\R^d$ if there is an expanding homothety $h$ of $\R^d$ (that is, $h(x)=rx+t$ for some $r>1$ and $t\in\R^d$) such that $M\subset h(A)$. A compact set $K$ is called a \emph{microset} of $A$ if there is a sequence $(M_i)$ of compact minisets of $A$ converging to $K$ in the Hausdorff metric. Finally, a compact set $A$ is called \emph{homogenous}, if every microset of $A$ is also a miniset of $A$. Examples of homogeneous sets include self-similar sets satisfying the strong separation condition for which the linear parts of the similarities contain no rotations, and closed sets invariant under $(x_1,\ldots,x_d)\mapsto (p x_1\bmod 1,\ldots,p x_d\bmod 1)$.

\begin{proposition}
 \label{prop:lip_inverse}
 Let $A\subset \R^d$ be a homogenous set and $V\in G(d,k)$. If $P_V|_A$ is injective, then it has Lipschitz inverse.
\end{proposition}
\begin{proof}
 Let $E$ be the set of $k$-planes for which $P_V|_A$ is not injective. In other words, $E$ is the set of $V\in G(d,k)$ that contain a direction determined by $A$. That is,
 \[
  E=\{V\in G(d,k) : \dir(A)\cap V\neq\varnothing \}.
 \]
 We are going to show that $E$ is compact. Then any $V\in G(d,k)\setminus E$ has an $\eps$-neighborhood of $k$-planes not in $E$, and the existence of the Lipschitz inverse again follows from the proof of \cite[Lemma 15.13]{Mattila1995}.

 We show first that $\dir(A)$ is compact. By compactness of $S^{d-1}$ it suffices to show that $\dir(A)$ is closed. Let $(e_i)_{i\in\N}\subset \dir(A)$, with $e_i\to e\in S^{d-1}$. For each $e_i$ we find $x_i$ and $y_i$ in $A$ determining the direction $e_i$. If $|x_i-y_i|\ge 1$ for infinitely many $i$, then by compactness there exist $x,y\in A$ determining the direction $e$. We can then assume that $|x_i-y_i|<1$ for all $i$. Consider the expanding homotheties $h_i(z)=(z-x_i)/|x_i-y_i|$ and note that $h_i(y_i)=e_i$ and $h_i(x_i)=0$. Thus $\{0,e_i\}$ is a miniset of $A$, and further, $\{0,e\}$ is a microset of $A$, by the convergence $e_i\to e$. Since every microset is also a miniset by assumption, there is a homothety $h$ such that $\{0,e\}\subset h(A)$,  thus $e\in \dir(A)$.

 Again by compactness of $G(d,k)$, it suffices to show that $E$ is closed. Let $(V_i)\subset E$, with $V_i \to V$. For each $V_i$, find $v_i\in \dir(A)\cap V_i$. By the compactness of $\dir(A)$, we can assume that $v_i$ converges to some $v\in \dir(A)$. On the other hand, $v\in V$ by the convergence $V_i\to V$. Thus $V\in E$.
\end{proof}
For homogenous sets, we can now improve Theorem \ref{thm:holder} as follows:
\begin{corollary}
 \label{cor:lipgraph}
 For a  homogenous set $A$ with $\udimb A < t < \frac{d-k}{2}$, there is an exceptional set $E\subset G(d,k)$, with $\dimh E \leq (k-1)(d-k)+ 2t$, so that for all $V\in G(d,k)\setminus E$, there is a Lipschitz function $f_V\colon V^\perp \to V$ whose graph covers $A$.
\end{corollary}
\begin{proof}
 Setting $E$ as in the proof of Proposition \ref{prop:lip_inverse}, it follows directly from Theorem \ref{thm:holder} that $\dimh E \leq  (k-1)(d-k)+ 2t$. By Proposition \ref{prop:lip_inverse} we can choose $f_V$ to be the extension of the inverse of $P_V|_A$.
\end{proof}

\begin{remark}
 \label{rem:lip_is_bets}
  In general, Lipschitz can not be replaced by $C^1$. Take for example a self-homothetic set $A$ in $\R^d$ satisfying strong separation condition and assume that $A$ is not contained in any (affine) hyperplane. Assume then that $A$ can be covered by a graph of a  $C^1$ function $f\colon V^\perp \to V$, where $V^\perp$ is $k$ dimensional. Fix any point $x\in A$ and consider tangents at $x$. Any tangent of the graph of $f$ is a $k$-plane, but any tangent of $A$ at $x$ is in general position, since it is essentially a copy of $A$. This is a contradiction, since any tangent of $A$ at $x$ should be covered by some tangent of the graph of $f$ at $x$.

  The fact that Lipschitz can not be replaced by $C^1$ also follows directly from \cite[Theorem 3.1]{Mattila1982}.
\end{remark}

\subsection{Spherical projections and H\"{o}lder coverings in polar coordinates}

Let us now consider a variant of the problem where orthogonal projections are replaced by spherical projections $P_h(x):=(h-x)/|h-x|$, for $h\in\R^d$, $x\in\R^d\setminus\{h\}$. In this context, we consider coverings of $A$ by graphs by writing $\R^d\setminus\{h\}$ in polar coordinates centered at $h$ (that is, we identify the point $h+t v$, where $v\in S^{d-1}$, with the pair $(v,t)\in S^{d-1}\times (0,\infty)$). The results are similar to those in the case $k=1$ in \S2, but the dimension estimates differ by an additional constant one. In the viewpoint of our proof, the reason is that if $P_h(a)=P_h(b)$, then $P_{h'}|_A$ is not injective for any $h'$ in the line determined by $a$ and $b$. So a pair of points forbids a one dimensional family, the line through $a$ and $b$, in a $d$ dimensional space $\R^d$, instead of a single point in a $d-1$ dimensional space $S^{d-1}=G(d,1)$.

\begin{theorem}
 \label{thm:visibility}
 Let $A\subset\mathbb{R}^d$ be a bounded set such that $\udimb(A)<t< (d-1)/2$ and let $2t < w < d-1$. Then the set of points $h\in \R^d\setminus A$ for which the set $A$ is not contained in the graph of a H\"older function $f_h:S^{d-1}\to (0,\infty)$ of exponent $1-\frac{2t}{w}$ has Hausdorff dimension at most $w+1$.
\end{theorem}
\begin{proof}
 The proof mimics that of Theorem \ref{thm:holder} so we only give a sketch. Fix $S>0$. To begin, observe that if $B, B'$ are balls of radius $r$, separated by a distance $R\ge r$, then the set of points $h \in B(0,S)$ which lie on a line through $B$ and $B'$ are contained in $T_{B,B'}(C_S r/R) \cap B(0,S)$, were
 \[
  T_{B,B'}(\delta) = \{ x\in\R^d : \dist(x,\ell(B,B'))\leq \delta \}
 \]
 is a tube, and $C_S$ is a constant depending only on $S$. Further, by dividing the tube into pieces of length $\delta$, it is easy to see that the $(w+1)$-dimensional Hausdorff content of $T_{B,B'}(\delta)\cap B(0,S)$ is at most $C'_{S}\delta^{w}$, where again $C'_S$ depends only on $S$. With this in hand, the proof continues as the case $k=1$ of proof of Theorem \ref{thm:holder}, with $M_n$ equal to the set points that lie on a line through a pair of balls in $\CC_n$, using the spherical projections $P_h$ instead of the projections to $V^\perp$, and considering the $(w+1)$-dimensional Hausdorff content of the points that are infinitely often in $M_n$. Since the result holds for every $S>0$, it holds also in all of $\R^d$.
\end{proof}
\begin{corollary}
 \label{cor:visibility}
 Let $A\subset\mathbb{R}^d$ be a bounded set such that $\udimb(A) = t < (d-1)/2$. Then the set of points $h\in \R^d\setminus A$ for which the spherical projection $P_h$ restricted to $A$ is not injective, has Hausdorff dimension at most $2t+1$.
\end{corollary}

\section{H\"{o}lder graphs and doubling measures}
A Borel measure $\mu$ on a metric space $X$ is said to be doubling, if there is a constant $C>1$ so that $\mu( B(x,2r) ) \leq C \mu(B(x,r))$ for any $x\in X$ and $r>0$ (for us, $X$ will always be a Euclidean space $\R^d$). A set is said to be \emph{thin} if it is of zero measure for all doubling measures of the ambient space. For example, a simple density point argument shows that upper porous sets are thin. We refer the reader to \cite{OjalaRajalaSuomala2012, WangWenWen2013} for further discussion and examples of thin sets. If a set is not thin, then we may say that it has positive doubling measure without specifying the measure.

As mentioned in the introduction, we can answer the question of whether all H\"older graphs are thin in the negative. By Theorem \ref{thm:holder} or \cite[Theorem 3.1]{HuntKaloshin1999} all that is needed is the existence of a set of small upper box dimension and positive doubling measure. The existence follows from \cite{KaenmakiRajalaSuomala2012}, as explained in Remark \ref{rem:tapio} below, but we prefer to exhibit a concrete self-similar example arising in \cite{GarnettKillipSchul2010} (in fact, this type of construction goes back even further to \cite{Wu1998}). For the reader's convenience, we briefly revise the construction. For $\delta>0$ and a probability vector $p=(\delta,1-2\delta,\delta)$, let $\mu$ be the associated ternary Bernoulli (self-similar) measure on the unit interval, extended $1$-periodically to the real line. Since the weights on the sides are equal, the resulting measure is doubling. Given a dimension $d\ge 2$, let $\nu$ be the $d$-fold product $\mu\times\cdots\times\mu$, which is again doubling. See \cite[\S 2.1]{GarnettKillipSchul2010} for the short proofs of these facts.

We use the convention that the ternary expansion of a number is the lexicographically smallest one, if there are two. Fix $n_1\in\N$ and choose $\delta$ so that $k_1:= 3\delta n_1 \in\N$, and let $k_\ell=\ell\cdot k_1$ for all $\ell\in\N$. Finally, set $S_L = \sum_{\ell=1}^L n_1\cdot \ell =n_1 \cdot L(L+1)/2$.

Now define $K$ to be the set of those points in $[0,1]$ whose ternary expansion contains at most $k_L$ zeros or twos between the positions $S_{L-1}+1$ and $S_L$ (we refer to $S_L$ as construction levels). In other words, let $x_j$ denote the $j$:th digit in the ternary expansion of $x$ and set
\[
 K := \left\{ x\in [0,1] :  x_j\in\{0,2\}\text{ for at most $k_L$ values of $j$ with $S_{L-1}< j \leq S_L$} \right\}.
\]
Note that the constructions of $K, \mu$ and $\nu$ depend on the parameter $\delta$.
\begin{lemma}
\label{lem:GKSdimension}
 The upper box dimension of $K$ can be made arbitrarily small
\end{lemma}
\begin{proof}
First of all, in the calculation of $\udimb$ it is enough to consider ternary intervals. At the construction levels $S_L$, we have the natural cover of $K$ by the ternary construction, and the number of intervals is at most $\exp\{ (k_1+\dots+k_L) (1+\log \delta^{-1}) \}$, as it follows from \cite[Equation (2.6)]{GarnettKillipSchul2010}. For $m\in\N$, set $j_m$ to be the difference from $m$ to the previous level of the construction. In other words, let $L_m$ be the largest integer so that $m-S_{L_m} =:j_m \geq 0$. Note that $j_m \leq (L_m + 1) n_1$. Also, $S_L\geq \tfrac12 L^2 n_1$ for large values of $L$ and so $m \geq\tfrac12 L_m^2 n_1$ for large values of $m$. Thus, letting $N(K,3^{-j})$ be the number of ternary intervals of side-length $3^{-j}$ that touch $K$, we have that
 \begin{align*}
  \frac{ \log N(K,3^{-m}) }{ \log 3^m }
  &\leq \frac{ \log [N(K,3^{-L_m}) 3^{j_m}] }{ m \log 3 }\\
  &\leq \frac{ \log[ \exp\{  S_{L_m} 3\delta (1+\log \delta^{-1}) \}] + j_m \log 3 }{ m \log 3 } \\
  &\leq \frac{3}{\log 3} \frac{S_{L_m}}{m} (\delta+\delta\log \delta^{-1}) + \frac{j_m}{ m } \\
  &\leq \frac{3}{\log 3} (\delta+\delta\log \delta^{-1}) + \frac{2(L_m + 1) n_1}{L_m^2 n_1}\\
  &\to  \frac{3}{\log 3} (\delta+\delta\log \delta^{-1})
 \end{align*}
 as $m\to\infty$. Since $(\delta+\delta\log \delta^{-1}) \to 0$ as $\delta\to 0$, for any $\eps$ we can choose $\delta$ so that $\udimb(K)\leq\eps$.
\end{proof}
It is (implicitly) shown in \cite{GarnettKillipSchul2010} that $\nu(K^d)>0$ (see, in particular, \cite[Equation (2.5)]{GarnettKillipSchul2010}). In fact, the measure can be made arbitrarily close to one by choosing $n_1$ large enough. Since $\udimb(K^d)\le d \udimb(K)$, we get the following corollary of \cite[Theorem 3.1]{HuntKaloshin1999} (and Theorem \ref{thm:holder}):

\begin{corollary}
 \label{cor:holderdoubling}
 For any $d\in\N_{\ge 2}$ and $\gamma<1$, there are a $\gamma$-H\"older function $f:\R^{d-1}\to \R$ and a self-similar doubling measure $\nu$ on $\R^d$, so that the graph of $f$ has positive measure with respect to $\nu$.
\end{corollary}

\begin{remark}
\label{rem:tapio}
 The existence of sets with small upper box dimension and positive doubling measure also follows from \cite{KaenmakiRajalaSuomala2012}, where it is shown that in any complete doubling metric space there are doubling measures giving full measure to a set of arbitrarily small packing dimension. In particular, in $\R^d$, for any $\eps>0$, there is a doubling measure $\mu$ and a bounded set $A\subset\R^d$ of positive measure, so that $\dimp (A) \leq \eps$. Since packing dimension can be defined in terms of upper box dimension, see for example \cite[Section 5.9]{Mattila1995}, one can choose $B\subset A$ so that $\mu(B) > 0$ and $\udimb (B) < 2\eps$. We thank T. Rajala for this remark.
\end{remark}

\bibliographystyle{abbrv}
\bibliography{References}

\end{document}